\documentclass[10pt,openany]{article}
\usepackage[a4paper,hmargin={4cm,4cm},vmargin={4cm,4cm}]{geometry}
\usepackage[a4paper,colorlinks=true,bookmarksopen=false,linkcolor=black,citecolor=red]{hyperref}
\usepackage[english]{babel}
\usepackage{amsrefs}
\usepackage{amssymb}
\usepackage{amsthm}
\usepackage{graphicx}
\usepackage{ifsym}
\usepackage{euscript}
\usepackage{fourier}
\usepackage{xcolor}
\definecolor{darkblue}{RGB}{0, 0, 127}
\def\emph#1{\textcolor{darkblue}{\textbf{\boldmath #1}}}

\def\ssf#1{\textsf{\footnotesize #1}}

\renewenvironment{itemize}
  {\begin{list}{$\triangleright$}{%
   \setlength{\parskip}{0mm}
   \setlength{\topsep}{.3\baselineskip}
   \setlength{\rightmargin}{0mm}
   \setlength{\listparindent}{0mm}
   \setlength{\itemindent}{0mm}
   \setlength{\labelwidth}{3.5ex}
   \setlength{\itemsep}{.3\baselineskip}
   \setlength{\parsep}{0mm}
   \setlength{\partopsep}{0mm}
   \setlength{\labelsep}{1ex}
   \setlength{\leftmargin}{\labelwidth+\labelsep}
   }}{%
   \end{list}\vskip-\parskip}

\newtheoremstyle{mio}% name
     {2\parskip}%      Space above
     {0mm}%      Space below
     {\sl}%         Body font
     {}%         Indent amount (empty = no indent, \parindent = para indent)
     {\bfseries}% Thm head font
     {}%        Punctuation after thm head
     {3mm}%     Space after thm head: " " = normal interword space;
     {\llap{\thmnumber{#2}\hskip2mm}\thmname{#1}\thmnote{\bfseries{} #3}}% Thm head spec (can be left empty, meaning `normal')

\newtheoremstyle{liscio}% name
     {2\parskip}%      Space above
     {0mm}%      Space below
     {}%         Body font
     {}%         Indent amount (empty = no indent, \parindent = para indent)
     {\bfseries}% Thm head font
     {}%        Punctuation after thm head
     {3mm}%     Space after thm head: " " = normal interword space;
     {\llap{\thmnumber{#2}\hskip2mm}\thmname{#1}\thmnote{\bfseries{}#3}}% Thm head spec (can be left empty, meaning `normal')
\newcounter{thm}

\theoremstyle{mio}
\newtheorem{theorem}[thm]{Theorem}

\newtheorem{proposition}[thm]{Proposition}
\newtheorem{lemma}[thm]{Lemma}

\newtheorem{definition}[thm]{Definition}

\theoremstyle{liscio}

\def\QED{\noindent\nolinebreak[4]\hfill\rlap{\ \ $\Box$}\medskip}
\renewenvironment{proof}[1][Proof]%
{\begin{trivlist}\item[\hskip\labelsep\textbf{#1}]}
{\QED\end{trivlist}}

\def\E{\exists}

\def\Aut{\textrm{Aut}}
\def\Autf{{\rm Aut\kern.16ex{f}}}

\def\IMP{\Rightarrow}

\def\imp{\rightarrow}
\def\iff{\leftrightarrow}
\def\equivL{\stackrel{\smash{\scalebox{.5}{\rm L}}}{\equiv}}
\def\C{\EuScript C}
\def\D{\EuScript D}
\def\Ll{\EuScript L}
\def\U{\EuScript U}
\def\O{\EuScript O}
\def\0{\varnothing}
\def\phi{\varphi}
\def\epsilon{\varepsilon}

\title{On the diameter of Lascar strong types\\ (after Ludomir Newelski)}
\author{Domenico Zambella}
\date{}

\begin{document}
      
\maketitle

\begin{center}
\textit{A contribution to Liber Amicorum Alberti\\[.5ex] dedicated to Albert Visser on the occasion of his 65th birthday}
\end{center}
\def\medrel#1{\parbox[t]{6ex}{$\displaystyle\hfil #1$}}
\def\ceq#1#2#3{\noindent\parbox[t]{25ex}{$\displaystyle #1$}\medrel{#2}{$\displaystyle #3$}}

\begin{abstract}\noindent
This is an exposition a theorem of mathematical logic which only assumes the notions of structure, elementary equivalence, and compactness (saturation).\\[1ex]
In~\cite{Newelski1} Newelski proved that type-definable Lascar strong types have finite diameter. This exposition is based on the proof in~\cite{Pelaez} up to a minor difference: the notion of \textit{weak c-free\/} of~\cite{NP} is replaced with the notion of \textit{non-drifting\/} that is introduced here.
\end{abstract}

\vspace{2ex}

\section{Introduction}

Few recent results in mathematical logic have a statement that is accessible to logicians outside a specific area. One of them is the theorem on the diameter of Lascar strong types. The theorem concerns a graph that can naturally be defined in any infinite structure. 

The problem can be presented in different ways that are equivalent. We choose the one that requires the fewer prerequisites. For a Galois-theoretical perspective, close to Lascar's original approach~\cite{Lascar}, we refer the reader to e.g.~\cite{Pelaez}. We assume the reader knows what a saturated model is and we fix one. This is denoted by \emph{$\U$\/} and will be our universe for the rest of the paper. We denote its cardinality by \emph{$\kappa$\/} which we assume to be uncountable and larger than the cardinality of the language. We also fix a set $A\subseteq\U$ of small cardinality, where \emph{small\/} means $<\kappa$. There would be no loss of generality in assuming $A=\0$. Indeed, $A$ is fixed throughout the following so it could be absorbed in the language and forgotten about. However, we display it all along. We denote by \emph{$L(A)$\/} the set of formulas with parameters in $A$. By \emph{$|L(A)|$\/} we denote the cardinality of the set of sentences in $L(A)$. This cardinality does not play a role in the proof and assuming $|L(A)|=\omega$ may help on the first reading.

Let $z$ be a tuple of variables of ordinal \emph{length\/} $|z|<\kappa$. Though the theorem is also interesting for infinite tuples, the length of $z$ does not play any role in the proof. Again, for a first reading one can assume $z$ is a single variable. If $a,b\in \U^{|z|}$ we write  $a\equiv_Ab$ if $\phi(a)\iff\phi(b)$ holds (in $\U$) for every $\phi(z)\in L(A)$. In words we say that $a$ and $b$ have the same type over $A$.

A \emph{definable set\/} is a set of the form $\phi(\U)=\big\{a\in\U^{|z|}\, :\, \phi(a)\big\}$ for some formula $\phi(z)\in L(\U)$. A \emph{type\/} is a set of formulas $p(z)\subseteq L(B)$ for some $B\subseteq\U$ of small cardinality. A \emph{type-definable set\/} is a set of the form $p(\U)$, that is, the intersection of $\phi(\U)$ for $\phi(z) \in p(z)$.

It may be useful (though not essential) to interpret this in topological terms. The sets $\phi(\U)=\big\{a\in\U^{|z|}\, :\, \phi(a)\big\}$ for $\phi(z)\in L(A)$ form a base for a topology. This topology is zero-dimensional and it is compact because $\U$ is saturated. It is never T$_0$ as any pair of tuples $a\equiv_A b$ have exactly the same neighborhoods, such a pair exists for cardinality reasons. However it is immediate that the topology induced on the quotient $\U^{|z|}/\!\equiv_A$ is Hausdorff (this is the so-called \textit{Kolmogorov quotient}). In this topology the closed sets are those of the form $p(\U)$ where $p(z)\subseteq L(A)$ is any type.

In what follows, by \emph{model\/} we understand an elementary substructure of $\U$ of small cardinality. The \emph{Lascar graph over $A$\/} has $\U^{|z|}$ as the set of vertices and an edge between all pairs of vertices $a,b\in \U^{|z|}$ such that  $a\equiv_Mb$ for some model $M$ containing $A$. We write \emph{$d_A(a,b)$\/} for the distance between $a$ and $b$ in the Lascar graph over $A$. Let us spell this out: $d_A(a,b)\le n$ if there is a sequence $a_0,\dots,a_n$ such that $a=a_0$, $b=a_n$, and $a_i\equiv_{M_i}a_{i+1}$ for some models $M_i$ containing $A$. We write \emph{$d_A(a,b)<\infty$\/} if $a$ and $b$ are in the same connected component of the Lascar graph over $A$. 

\begin{definition}\label{tipoforteLascarediametro} For every $a\in\U^{|z|}$

\begin{itemize}
\item[1.] $a\equivL_Ab$ \ if \ $d_A(a,b)<\infty$;
\item[2.] $\Ll(a/A)\ =\ \big\{b\ :\ a\equivL_Ab\big\}.$
\end{itemize}
We call $\Ll(a/A)$ the  Lascar strong type of $a$ over $A$. If $a\equivL_Ab$ we say that $a$ and $b$ have the same Lascar strong type over $A$.
\end{definition}
\smallskip

We are ready to state Newelski's theorem which we prove in the next section.
%\begin{restatable}{theorem}{maintheorem}\label{thm_main}
\begin{theorem}\label{thm_main}
For every $a\in\U^{|z|}$ the following are equivalent
\begin{itemize}
 \item[1.] $\Ll(a/A)$ is type-definable;
 \item[2.] $\Ll(a/A)\ =\ \big\{c\ :\ d_A(a,c)<n\big\}$ for some $n<\omega$.
\end{itemize}
%\end{restatable}
\end{theorem}

Newelski's original proof has been simplified over the years. Most proofs have a definite topological dynamics flavor (the liaison with topological dynamics was clarified in~\cite{Newelski2}). Below we give a streamlined version of the proof in~\cite{Pelaez} (see also~\cite[Theorem~9.22]{Casanovas}).

More recent contributions to the subject have investigated the descriptive set theoretic complexity of the relation of having the same Lascar strong types. This is beyond the scope of this short note so we refer the interested reader to~\cite{KPS}, \cite{KMS} and~\cite{KM}.

It is interesting to note that if $\Ll(a/A)$ is type-definable for every $a\in\U^{|z|}$ then the equivalence relation $\equivL_A$ is also type-definable. This might be surprising at first, so we sketch a proof below  (not required for the main theorem).

The equivalence relation $\equivL_A$ is \emph{invariant\/} over $A$, that is, invariant over automorphisms that fix $A$. Its equivalence classes, are \emph{Lascar invariant\/} over $A$, that is, invariant over automorphisms that fix some model containing $A$. There are at most $2^{2^{|L(A)|}}$ sets that are Lascar invariant over $A$. Then  $\equivL_A$ is a \emph{bounded\/} equivalence relation, that is, it has $<\kappa$ equivalence classes. It is not difficult to verify that  $\equivL_A$ is the finest bounded equivalence relation invariant over $A$.
 
Let $e_A(x,z)\subseteq L(A)$ be the union of all types over $A$ that define a bounded equivalence relation. It suffices to prove that if $\Ll(a/A)$ is type-definable then $\Ll(a/A)=e_A(\U,a)$. In fact, if this holds for every $a\in\U^{|z|}$, then $e_A(x,z)$ defines $\equivL_A$.
 
If $\Ll(a/A)$ is type-definable, by Theorem~\ref{thm_main} and Proposition~\ref{prop_Lascar_distance_type_def} it is defined by the type $d_A(a,z)<n$ for some $n$. The same is true for every $b\equiv_A a$. Hence $d_A(x,z)<n$ defines an equivalence relation whose restriction to $\O(a/A)=\{b:b\equiv_A a\}$ has boundedly many classes. By a well-known fact (see for example \cite[Proposition~5.11]{Simon}), there is a bounded equivalence relation type-definable over $A$ that coincides with $d_A(x,z)<n$ on $\O(a/A)$. Then $e_A(\U,a)\subseteq\Ll(a/A)$ follows. The converse inclusion is trivial.
 
In his seminal paper~\cite{Lascar} Lascar asked for (not literally but in an equivalent way) examples where the relation of having the same Lascar strong type is not type-definable. By the theorem above this is equivalent to asking for structures where the diameter of a connected component of the Lascar graph is infinite. In tame structures, like stable and simple ones, the diameter is always finite. The first example with infinite diameter was constructed by Ziegler~\cite{CLPZ} and later more natural examples were found~\cite{CP}. 

\section{Lascar strong automorphisms}

It may not be immediately obvious that the relation $d_A(z,y)\le n$ is type-definable. From this the easy direction of the main theorem follows.

\begin{proposition}\label{prop_Lascar_distance_type_def}
For every $n<\omega$ there is a type $p_n(z,y)\subseteq L(A)$ equivalent to \mbox{$d_A(z,y)\le n$}. 

\end{proposition}
\begin{proof}
In a saturated structure types are closed under existential quantification, therefore it suffices to prove the proposition with $n=1$. Let $\lambda=|L(A)|$ and let $w=\langle w_i:i<\lambda\rangle$ be a tuple of distinct variables. Then $p_1(z,y)\;=\;\E w\;p(w,z,y)$ where

\ceq{\hfill p(w,z,y)}{=}{q(w)\ \cup\ \Big\{\phi(z,w)\iff\phi(y,w)\ :\ \phi(z,w)\in L(A)\Big\}}

\noindent and $q(w)\subseteq L(A)$ is a consistent type with the property that all its realizations enumerate a model containing $A$.

Now we only need to prove that such a type exists. Let $\langle\psi_i(x,w_{\restriction i}):i<\lambda\rangle$ be an enumeration of the formulas in $L_{x,w}(A)$, where $x$ is a single variable. Let 

\ceq{\hfill q(w)}{=}{\big\{\E x\;\psi_i(x,w_{\restriction i})\;\imp\;\psi_i(w_i,w_{\restriction i})\ :\ i<\lambda\big\}.}

\noindent Any realization of $q(w)$ satisfy the Tarski-Vaught test therefore it enumerates a model containing $A$. Vice versa it is clear that we can realize $q(w)$ in any model containing $A$.
\end{proof}

We write \emph{Aut\kern.16ex$(\U/A)$} for the set of automorphisms of $\U$ that fix $A$.
We write \emph{Aut\kern.16ex{f}$(\U/A)$\/} for the subgroup of $\Aut(\U/A)$ generated by the automorphisms that fix some model $M$ containing $A$. The ``f'' in the symbol stands for \textit{fort}, the French for \textit{strong}. It is immediate to verify that $\Autf(\U/A)$ is a normal subgroup of $\Aut(\U/A)$.

Recall that saturated models are homogeneous, hence any $a\equiv_Bb$ are conjugated over $B$, that is, there is an $f\in\Aut(\U/B)$ such that $fa=b$. Then it is easy to verify that $a\equivL_Ab$ \ if and only if $fa=b$ for some $f\in\Autf(\U/A)$. The following notions apply generally to any group $G$ acting on some set $X$ and and to any set $\D\subseteq X$. Below we always have $G=\Autf(\U/A)$ and $X=\U^{|z|}$.  We say that \emph{$\D$ is drifting\/} if for every finitely many $f_1,\dots,f_n\in G$ there is a $g\in G$ such that $g[\D]$ is disjoint from all the $f_i[\D]$. We say that \emph{$\D$ is quasi-invariant\/} if for every finitely many $f_1,\dots,f_n\in G$ the sets  $f_i[\D]$ have non-empty intersection. Note parenthetically that $\D$ is quasi-invariant if and only if $\neg\D=X\smallsetminus\D$ is not c-free in the sense of \cite{Pelaez} or not generic in the sense of~\cite{KMS}. We say that a formula or a type is drifting or quasi-invariant if the set it defines is.

The union of drifting sets need not be drifting. However, the following lemma says it cannot be quasi-invariant.

\begin{lemma}\label{lem_newelski}
The union of finitely many drifting sets in not quasi-invariant.
\end{lemma}

\begin{proof}
\def\medrel#1{\parbox[t]{6ex}{$\displaystyle\hfil #1$}}
\def\ceq#1#2#3{\noindent\parbox[t]{30ex}{$\displaystyle #1$}\medrel{#2}{$\displaystyle #3$}}

It is convenient to prove an apparently more general claim. If $\D_1,\dots,\D_n$ are all drifting and $\Ll$ is such that for some finite $F\subseteq G$\medskip 

\ceq{\sharp\hfill\Ll}{\subseteq}{\bigcup_{f\in F}f[\D_1\cup\dots\cup\D_n],} 

\noindent then $\Ll$ is not quasi-invariant. (The statement is slightly awkward since a superset of a quasi-invariant set must be quasi-invariant.)

The claim is vacuously true for $n=0$. Let $n$ be positive, let $\C=\D_1\cup\dots\cup\D_{n-1}$, and assume the claim holds for $n-1$. Since $\D_n$ is drifting there is a $g\in G$ such that $g[\D_n]$ is disjoint from $f[\D_n]$ for every $f\in F$, which implies that\medskip 

\ceq{\hfill\Ll\cap g[\D_n]}{\subseteq}{\bigcup_{f\in F}f[\C].} 

\noindent Hence for every $h\in G$ there holds\medskip 

\ceq{\hfill h g^{-1}[\Ll]\cap h[\D_n]}{\subseteq}{\bigcup_{f\in F}hg^{-1}f[\C].} 

\noindent Rewriting $\sharp$ as

\ceq{\hfill\Ll}{\subseteq}{\bigcup_{f\in F}f[\C]\ \cup\ \bigcup_{h\in F}f[\D_n],}

\noindent we observe that\medskip

\ceq{\hfill \Ll \cap\bigcap_{h\in F}h g^{-1}[\Ll]}{\subseteq}{\bigcup_{f\in F}f[\C]\ \cup\ \bigcup_{f\in F}hg^{-1}f[\C].}\smallskip

\noindent By the induction hypothesis, the r.h.s.\@ cannot be quasi invariant. Hence neither is $\Ll$, proving the claim and with it the lemma.
\end{proof}

The following is a consequence of Baire's category theorem. We sketch a proof for the convenience of the reader.

\begin{lemma}\label{Baire}
Let $p(x)\subseteq L(B)$ and $p_n(x)\subseteq L(A)$, for $n<\omega$, be consistent types such that\smallskip

\ceq{\ssf{1.}\hfill p(x)}{\imp}{\bigvee_{n<\omega}p_n(x)}

\noindent Then there is an $n<\omega$ and a formula $\phi(x)\in L(A)$ consistent with $p(x)$ such that \smallskip

\ceq{\ssf{2.}\hfill p(x)\wedge\phi(x)}{\imp}{p_n(x)}
\end{lemma}
\begin{proof}
Negate \ssf{2} and choose inductively for every $n<\omega$ a formula $\psi_n(x)\in p_n(x)$ such that $p(x)\wedge\neg\psi_0(x)\wedge\dots\wedge\neg\psi_n(x)$ is consistent. By compactness, we contradict \ssf{1}.
\end{proof}

Finally we can prove the Theorem~\ref{thm_main} which we restate for convenience.
\setcounter{thm}{1}
\begin{theorem}
For every $a\in\U^{|z|}$ the following are equivalent
\begin{itemize}
 \item[1.] $\Ll(a/A)$ is type-definable;
 \item[2.] $\Ll(a/A)\ =\ \big\{c\ :\ d_A(a,c)<n\big\}$ for some $n<\omega$.
\end{itemize}
\end{theorem}
%\maintheorem*

\begin{proof}
Implications \ssf{2}$\IMP$\ssf{1} holds by Proposition~\ref{prop_Lascar_distance_type_def}. We prove  \ssf{1}$\IMP$\ssf{2}. Suppose $\Ll(a/A)$ is type-definable, say by the type $l(z)$. Let $p(z,y)$ be some consistent type (to be defined below) such that and $p(z,y)\imp l(z)\wedge l(y)$. Then, in particular\smallskip

\ceq{\hfill p(z,y)}{\imp}{\bigvee_{n<\omega}d_A(z,y)<n.}

By Proposition~\ref{prop_Lascar_distance_type_def} and Lemma~\ref{Baire}, there is some $n<\omega$ and some $\phi(z,y)\in L(A)$  consistent with $p(z,y)$ such that\smallskip

\ceq{\sharp_1\hfill p(z,y)\wedge\phi(z,y)}{\imp}{d_A(z,y)<n}.\smallskip

\noindent Below we define $p(z,y)$ so that for every $\psi(z,y)\in L(A)$\smallskip

\ceq{\sharp_2\hfill p(z,a)\wedge\psi(z,a)}{}{} is non-drifting whenever it is consistent. \smallskip

\noindent Drifting and quasi-invariance are relative to the action of $\Autf(\U/A)$ on $\U^{|z|}$. Then, in particular, $p(z,a)\wedge\phi(z,a)$ is non-drifting and the theorem follows. In fact, let $a_0,\dots, a_k\in\Ll(a/A)$ be such that every set $p(\U,c)\cap\phi(\U,c)$ for $c\in\Ll(a/A)$ intersects some $p(\U, a_i)\cap\phi(\U,a_i)$. Let $m$ be such that $d_A(a_i,a_j)\le m$ for every $i,j\le k$. From $\sharp_1$ we obtain that $d_A(a,c)\le m+2n$. As $c\in\Ll(a/A)$ is arbitrary, the theorem follows.

The required type $p(z,y)$ is union of a chain of types $p_\alpha(z,y)$ defined as follows\smallskip

\ceq{\hfill p_{0}(z,y)}{=}{l(z)\ \cup\ l(y);}\smallskip

\ceq{\sharp_3\hfill p_{\alpha+1}(z,y)}{=}{p_\alpha(z,y)\ \cup\ \Big\{\neg\psi(z,y)\in L(A)\ :\ p_\alpha(z,a)\wedge\psi(z,a) \textrm{ is drifting}\Big\};}

\ceq{\hfill p_{\alpha}(z,y)}{=}{\bigcup_{n<\alpha}p_{n}(z,y)}\quad for limit $\alpha$.

\noindent Clearly, the chain stabilizes at some stage $\le|L(A)|$ yielding a type which satisfies $\sharp_2$. So we only need to prove consistency. We prove that $p_\alpha(z,a)$ is quasi-invariant (so, in particular, consistent). Suppose that $p_n(z,a)$ is quasi-invariant for every $n<\alpha$ but, for a contradiction, $p_\alpha(z,a)$ is not. Then for some $f_1,\dots,f_k\in\Autf(\U/A)$

\ceq{\hfill p_\alpha(z,a)}{\cup}{\bigcup^k_{i=1}p_\alpha(z,f_ia)}

\noindent is inconsistent. By compactness there is some $n<\alpha$ and some $\psi_i(z,y)$ as in $\sharp_3$ such that

\ceq{\hfill p_n(z,a)}{\imp}{\neg\bigwedge^m_{j=1}\bigwedge^k_{i=1}\neg\psi_j(z,f_ia)}

As $p_n(z,a)$ is quasi-invariant, from Lemma~\ref{lem_newelski} we obtain that $p_n(z,f_ia)\wedge\psi_j(z,f_ia)$ is non-drifting for some $i,j$. Clearly we can replace $f_ia$ with $a$, then this contradicts the construction of $p_\alpha(z,y)$ and proves the theorem.
\end{proof}

We are indebted to the anonymous referee for many useful comments and for a neat proof of Lemma~\ref{lem_newelski}.

\vfil
\hfill\begin{minipage}{28ex}
Domenico Zambella\\
Dipartimento di Matematica\\
Universit\`a di Torino\\
via Calrlo Alberto 10\\
10123 Torino\\
\texttt{domenico.zambella@unito.it}
\end{minipage}

\end{document}